%% file: lutz-ENAHA.tex
\documentclass[oneside,reqno,11pt]{amsart}

\usepackage{lutz-ENAHA}

\begin{document}
	
	\title{Electrical networks and hyperplane arrangements}
	
	\author{Bob Lutz}
	\address{Department of Mathematics, University of Michigan, Ann Arbor, MI, USA}
	\email{boblutz@umich.edu}
	\thanks{Work of the author was partially supported by NSF grants DMS-1401224 and DMS-1701576.}
	
	\date{\today}


\begin{abstract}
	\input{sections/abstract}
\end{abstract}

\maketitle

\input{sections/intro2}
\input{sections/maindef}
\input{sections/combo}
\input{sections/super}

\input{sections/master}
\appendix
\input{sections/galois}
\input{sections/thanks}

\bibliographystyle{abbrv}
\bibliography{lutz-ENAHA}
\end{document}

%% file: sections/abstract.tex
		This paper studies \emph{Dirichlet arrangements}, a generalization of graphic hyperplane arrangements arising from electrical networks and order polytopes of finite posets. We generalize descriptions of combinatorial features of graphic arrangements to Dirichlet arrangements, including characteristic polynomials and supersolvability. We apply these results to visibility sets of order polytopes and fixed-energy harmonic functions on electrical networks.

%% file: sections/intro2.tex
\section{Introduction}
\label{sec:intro}

The \emph{graphic arrangement} associated to a graph $\g=(V,E)$ is the set $\Ag$ of hyperplanes in $\R^V$ given by $x_i=x_j$ for all $ij\in E$. Graphic arrangements are fundamental in the study of hyperplane arrangements due to the relative ease of translating combinatorial and topological data from $\Ag$, often intractable for general arrangements, into graph-theoretic terms.

This paper studies \emph{Dirichlet arrangements}, a generalization of graphic arrangements arising from electrical networks and order polytopes of finite posets.
%We show that the main combinatorial features of Dirichlet arrangements generalize those of graphic arrangements in interesting and user-friendly ways. As applications, we characterize harmonic functions on electrical networks as critical points of master functions in the sense of Varchenko, and we give a simple formula for the number of visibility sets of an order polytope.
Let $\g$ be a finite connected undirected graph with no loops or multiple edges. Let $\B\subset V$ be a set of $\geq 2$ vertices called \emph{boundary nodes}, no two of which are adjacent. Let $ u:\B\to \R$ be injective. We think of $\g$ as a network of linear resistors with voltages $ u$ imposed on the boundary nodes. The \emph{Dirichlet arrangement} $\Au$ is the set of intersections of hyperplanes in the graphic arrangement $\Ag$ with the affine subspace
\begin{equation}
\{x\in \R^V:x_j =  u(j)\mbox{ for all }j\in \B\}.
\label{eq:affsub}
\end{equation}
The Dirichlet arrangement $\Au$ is not a genuine restriction of $\Ag$, since \eqref{eq:affsub} is not an intersection of elements of $\Ag$. %In general there is little combinatorial resemblance between an arrangement $\AA$ and an arbitrary affine slice of $\AA$.
However we show that $\Au$ preserves a good deal of graphic structure.

\begin{thm}
	Let $\mg$ be the graph obtained from $\g$ by adding an edge between each pair of boundary nodes. The following hold:
	\begin{enumerate}[(i)]
		\item The intersection poset $L(\Au)$ is the order ideal of $L(\Ag)$ consisting of all boundary-separating connected partitions of $\g$
		\item The characteristic polynomial of $\Au$ is the quotient of the chromatic polynomial of $\mg$ by a falling factorial
		\item The bounded chambers of $\Au$ correspond to the possible orientations of current flow through $\g$ respecting the voltages $u$ and in which the current flowing through each edge is nonzero.
	\end{enumerate}
	\label{thm:maincomb}
\end{thm}

Each part of Theorem \ref{thm:maincomb} generalizes a key theorem on graphic arrangements. As corollaries, we obtain a formula for the number of orientations in part (iii), and we show that the coefficients of a chromatic polynomial remain log-concave after ``modding out'' by a clique of the graph.

We also characterize supersolvable Dirichlet arrangements. %Supersolvable arrangements are desirable in part because their characteristic polynomials and related algebras factor nicely.
Stanley \cite{stanley1972} showed that the graphic arrangement $\Ag$ is supersolvable if and only if the graph $\g$ is chordal (or triangulated). We prove the following theorem, building on results of \cite{mu2015,stanley2015,suyama2018}.

\begin{thm}
	The Dirichlet arrangement $\Au$ is supersolvable if and only if the graph $\mg$ from Theorem \ref{thm:maincomb} is chordal.
	\label{thm:supintro}
\end{thm}

This answers a question posed by Stanley \cite{stanleytalk}, who asked if there is a characterization of supersolvable arrangements $\Au$ analogous to the graphic case (see Section \ref{sec:super}).

For an application of our results, let $P$ be a finite poset and $O$ the convex polytope in $\R^P$ of all order-preserving functions $P\to [0,1]$. Here $O$ is called the \emph{order polytope} of $P$ \cite{stanley1986}. %The \emph{visibility arrangement} of $O$ is the set $\AA(O)$ of affine spans of facets of $O$. The chambers of $\AA(O)$ correspond to
Consider the sets of facets of $O$ visible from different points in $\R^P$, called \emph{visibility sets} of $O$. In general not all visibility sets of $O$ are visible from far away, since certain obstructions are eliminated by viewing $O$ ``from infinity.'' Write $\alpha(\g)$ and $\beta(\g)$ for the number of acyclic orientations and the beta invariant, resp., of $\g$.

\begin{cor}
	Let $O$ be the order polytope of a finite poset. There is a graph $\Delta$ such that $O$ has exactly $\frac{1}{2}\alpha(\Delta)$ visibility sets, of which exactly $\frac{1}{2}\alpha(\Delta)-\beta(\Delta)$ are visible from far away.
	\label{cor:vis}
\end{cor}

Another application involves electrical networks with fixed boundary voltages. The pair $(\g, u)$ represents such a network if we consider the edges $E$ as resistors of equal conductance. By assigning complex edge weights $\con\in \C^E$ we can represent networks involving more general electrical components, such as RLC circuits.

For generic $\con\in \C^E$ the triple $(\g, u,\con)$ determines a unique harmonic function $h:V\to \C$ extending $ u$. In this scenario the \emph{energy} dissipated by a resistor $ij\in E$ is given by
\begin{equation}
\eng_{ij}=\con_{ij}(h(i)-h(j))^2.
\end{equation}	
Given a network $(\g, u)$ and fixed energies $\eng\in \C^E$, it is natural to ask which conductances $\con\in \C^E$ produce the energies $\eng$. Abrams and Kenyon \cite{abrams2016} posed the equivalent problem of describing the set of harmonic functions associated to these $\con$, called \emph{$\eng$-harmonic functions} on $(\g, u)$.

We describe the $\eng$-harmonic functions on $(\g,u)$ as critical points of \emph{master functions} of $\Au$ in the sense of Varchenko \cite{varchenko2003}. Broadly, master functions generalize logarithmic barrier functions, and their critical points generalize analytic centers of systems of linear inequalities.

\begin{thm}
	The $\eng$-harmonic functions on $(\g, u)$ are the critical points of the master function of $\Au$ with weights $\eng$.
	\label{thm:cpchar}
\end{thm}

Theorem \ref{thm:cpchar} connects electrical networks, a subject with a vast literature \cite{brooks1940, curtis2000, postnikov2006, seshu1961, smale1972, weyl1923}, to critical points of master functions, an active area of research with applications to Lie algebras, physics, integrable systems, and algebraic geometry \cite{cohen2011, huh2013, mukhin2005, varchenko2011, varchenko2014}. We obtain results of Abrams and Kenyon \cite{abrams2016} as corollaries of Theorem \ref{thm:cpchar}. Combining these results with Theorem \ref{thm:maincomb} yields the following.

\begin{cor}
	For generic $\eng$, the number of $\eng$-harmonic functions is $\beta(\mg)/(|\B|-2)!$.
	\label{cor:eharmcount}
\end{cor}

The paper is organized as follows. In Section \ref{sec:maindef} we establish basic properties and examples of Dirichlet arrangements. In Section \ref{sec:combo} we prove Theorem \ref{thm:maincomb}. In Section \ref{sec:super} we prove Theorem \ref{thm:supintro} and discuss the relationship of Dirichlet arrangements to previous work \cite{mu2015, stanley2015, suyama2018}. In Section \ref{sec:master} we prove Theorem \ref{thm:cpchar}. In the appendix we exhibit an action of $\Gal(\Qtr/\Q)$ on the critical points of any master function with positive rational weights, where $\Qtr$ is the field of totally real numbers.

%% file: sections/maindef.tex
\section{Dirichlet arrangements}
\label{sec:maindef}
	
	Given a base field $\KK$, an \emph{arrangement} $\AA$ in $\KK^d$ is a finite set of affine hyperplanes of $\KK^d$. We consider each arrangement $\AA$ in $\KK^d$ to be equipped with a set $\{f_H : H\in\AA\}$ of affine functionals $f_H:\KK^d\to\KK$ such that $H=\ker f_H$ for all $H\in \AA$. The $f_H$ are called \emph{defining functions} of $\AA$. We write
	\[Q(\AA)=\prod_{H\in \AA} f_H(x).\]
	
	An arrangement $\AA$ is \emph{defined over} a subring $S\subset \KK$ if in the standard basis of $\KK^d$ all coefficients of all defining functions $f_H$ belong to $S$. We write 
	\[T(\AA)=\bigcap_{H\in \AA} H\]
	if $\AA$ is nonempty and $T(\emptyset)=\KK^d$. We also write $M(\AA)=\KK^d\setminus \bigcup_{H\in \AA} H$. When $\KK=\R$, the connected components of $M(\AA)$ are called the \emph{chambers} of $\AA$. The arrangement $\AA$ is called \emph{central} if $T(\AA)$ is nonempty and \emph{essential} if the normal vectors of $\AA$ span $\KK^d$. The \emph{cone} $c\AA$ over $\AA$ is the central arrangement in $\KK^{d+1}$ defined by
	\[Q( c\AA) = x_0\prod_{H\in \AA} {f^h_H(x_0,\ldots,x_d)},\]
	where $f^h_H$ is the homogenization of $f_H$ with respect to the new variable $x_0$.

	\subsection{Dirichlet arrangements}
	
	By a \emph{graph} we will mean one that is finite, connected and undirected with no loops or multiple edges. Denote by $\g=(V,E)$ a graph on $d$ vertices and $k$ edges. The \emph{graphic arrangement} $\Ag$ of $\g$ over a field $\KK$ is the arrangement in $\KK^d$ defined by
	\[Q(\Ag)=\prod_{ij\in E} (x_i-x_j).\]
	
	Fix a set $\B\subsetneq V$ of $\geq 2$ vertices, no two of which are adjacent, and an injective function $u:\B\to \KK$. We call $\B$ the \emph{boundary} of $\g$ and $\inte=V\setminus \B$ the \emph{interior} of $\g$. We call $ u$ the \emph{boundary data} and the scalars $ u(j)\in\KK$ the \emph{boundary values}. The elements of $\B$ are called \emph{boundary nodes}. Write $m=|\B|$ and $n=|\inte|$, so $d=m+n$. Whenever the vector spaces $\KK^d$ and $\KK^n$ appear, we consider their coordinates to be indexed by $V$ and $\inte$, resp.
	
	\begin{mydef}
		Let $\Au$ be the arrangement in
		\[\XX=\{x\in \KK^d : x_j =  u(j)\mbox{ for all }j\in \B\}\cong \KK^n\]
		of hyperplanes $H\cap \XX$ for all $H\in \Ag$. An arrangement $\AA$ is \emph{Dirichlet} if $\AA=\Au$ for some $(\g, u)$.
		\label{def:Au}
	\end{mydef}

	This definition can be relaxed to include $m=0$, in which case $\Au=\Ag$ is graphic, and $m=1$, in which case $\Au$ has the same underlying combinatorics as $\Ag$. We restrict our attention to $m\geq 2$ in order to distinguish our results from the graphic cases. For example, graphic arrangements are central but not essential, but we have the following for Dirichlet arrangements.
	
	\begin{prop}
		Dirichlet arrangements are essential but not central.
		
		\begin{proof}
			For each $e\in E$ let $H_e$ be the corresponding element of $\Au$ with normal vector $v_e$ of the form $x_i-x_j$ or $x_i- u(j)x_j$. Since $\g$ is connected, for any $i\in \inte$ there is a path $P\subset E$ with one endpoint $i$ and the other endpoint in $\B$. We have $\sum_{e\in P} v_e = x_i$, replacing some $v_e$ with $-v_e$ if necessary. It follows that the normal vectors span $\KK^n$, so $\Au$ is essential. Since $\g$ is connected, there is a path $Q\subset E$ of between distinct boundary nodes $j$ and $j'$. If $x\in \bigcap_{e\in Q} H_e$, then $ u(j)=x_j=x_{j'}=  u(j')$, a contradiction. Hence $\Au$ is not central.
		\end{proof}
		
		\label{prop:centess}
	\end{prop}
	
	Definition \ref{def:Au} can also be modified to accommodate repeated boundary values and edges between boundary nodes. In case of repeated boundary values, one can identify all vertices on which $ u$ takes the same value, removing any duplicate edges. In case $\B$ is not an independent set, one can simply remove all edges between boundary nodes, assuming that the resulting graph is connected. Then define $\Au$ as in Definition \ref{def:Au}.  
	
	For the remainder of the paper we will assume that $\KK=\R$. We think of $\Au$ as an $n$-dimensional affine slice of $\Ag$, where for all $i\in \inte$ the coordinates $x_i$ of $\XX\cong\R^n$ are inherited from $\R^d$, and for all $j\in \B$ the coordinate $x_j$ is specialized to the boundary value $ u(j)$. 
	
	\begin{eg}[Wheatstone bridge]
		Consider the graph $\g$ on the left side of Figure \ref{fig:dir}, where the boundary nodes $j_1$ and $j_2$ are marked by white circles. The pair $(\g,\B)$ is sometimes called a \emph{Wheatstone bridge} after the work of C.~Wheatstone \cite{wheatstone1843}. A Wheatstone bridge can also be represented by a circuit diagram, as in Figure \ref{fig:dir}; here, jagged edges denote resistors, and the symbol on top denotes a battery between the boundary nodes.
		
		Fix boundary values $ u(j_1)=1$ and $ u(j_2)=-1$. This corresponds to placing a 2-volt battery between the boundary nodes. Writing  $\inte= \{i_1,i_2\}$, the Dirichlet arrangement $\Au$ is defined by
		\[Q(\Au)=(x_{i_1}^2-1)(x_{i_2}^2-1)(x_{i_1}-x_{i_2}).\]
		The bounded chambers of $\Au$ are the open triangles shaded on the right-hand side of Figure \ref{fig:dir}.  
		
		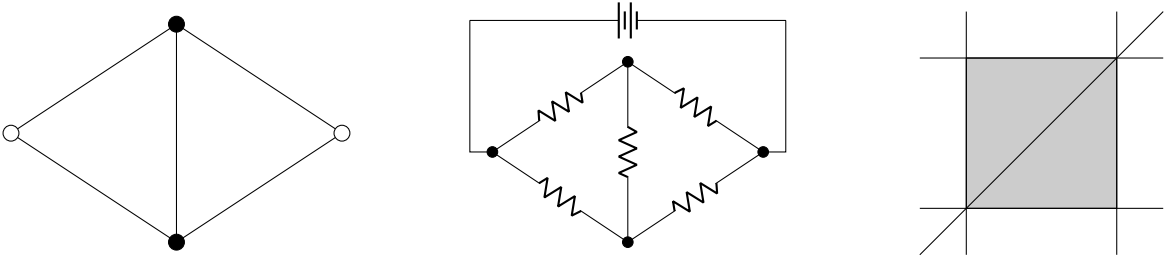
\begin{figure}[ht]
			\centering
			\input{tikz/dir}
			\caption{Left to right: a Wheatstone bridge $(\g,\B)$; the associated circuit diagram; and a corresponding Dirichlet arrangement $\Au$.}
			\label{fig:dir}
		\end{figure}
		
		\label{eg:dir}
	\end{eg}

\begin{eg}[Slices of braid arrangements]
	Suppose that $\inte$ is a clique, and that every vertex in $\B$ is adjacent to every vertex in $\inte$. We denote this graph with specified boundary by 
	\[\comp{m}{n}=(\g,\B).\]
	For instance, the Wheatstone bridge in Example \ref{eg:dir} is $\comp{2}{2}$. The case $\comp{5}{4}$ is illustrated in Figure \ref{fig:compjoin} with boundary nodes marked by white circles. Every Dirichlet arrangement is a subset of some $\Au$, where $(\g,\B)=\comp{m}{n}$.
	
	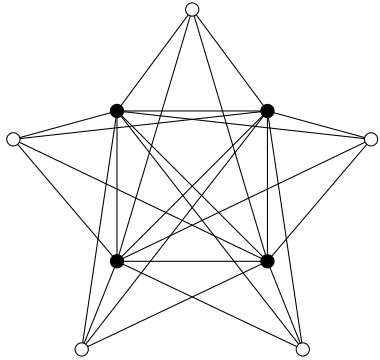
\begin{figure}[ht]
		\centering
		\input{tikz/compjoin}
		\caption{The network $\comp{5}{4}$ with boundary nodes marked in white.}
		\label{fig:compjoin}
	\end{figure}
	\label{eg:compjoin}
\end{eg}

\begin{eg}[Visibility arrangements of order polytopes]
	Let $P$ be a finite poset. The \emph{order polytope} $\mathcal{O}(P)$ of $P$ is the set of all order-preserving functions $P\to [0,1]$. Clearly $\mathcal{O}(P)$ is a convex polytope in $\R^P$. The \emph{visibility arrangement} $\operatorname{vis}(\mathcal{O}(P))$ of $\mathcal{O}(P)$ is the arrangement in $\R^P$ whose elements are the affine spans of all facets of $\mathcal{O}(P)$. It is so named because the chambers of $\operatorname{vis}(\mathcal{O}(P))$ correspond to the sets of facets of $\mathcal{O}(P)$ visible from different points in $\R^P$. Notice that the unbounded chambers of $\operatorname{vis}(\mathcal{O}(P))$ correspond to the sets of facets visible from far away.
	
	Consider the Hasse diagram $H$ of $P$ as a graph, so that the elements of $P$ are the vertices of $H$. Let $\g$ be the graph obtained by adding 2 vertices $j_0$ and $j_1$ to $H$, with $j_0 \sim i$ if $i$ is minimal in $P$ and $j_1\sim i$ if $i$ is maximal in $P$. Let $\B=\{j_0,j_1\}$, and let $ u:\B\to \R$ be given by $ u(j_0)=0$ and $ u(j_1)=1$. Then $\Au=\operatorname{vis}(\mathcal{O}(P))$ (see \cite[Theorem 4]{stanley2015}).
	\label{eg:vis}
\end{eg}

\begin{eg}[Linear order polytope]
	Let $P=\{1,\ldots,\ell\}$ with the usual linear ordering. The weakly increasing maps $P\to [0,1]$ correspond to points $x\in\R^\ell$ with $0\leq x_1 \leq \cdots \leq c_\ell \leq 1$. Thus the order polytope $\mathcal{O}(P)$ is an $\ell$-simplex in $\R^\ell$. Every nonempty subset of the $\ell+1$ facets of $\mathcal{O}(P)$ is a visibility set; by Corollary \ref{cor:vis} we must have $\frac{1}{2}\alpha(\mg)=2^{\ell+1}-1$, where $\mg$ (as defined in Example \ref{eg:vis}) is a cycle graph on $\ell+2$ vertices. The set of all facets is only visible from the interior of $\mathcal{O}(P)$, and is the only set not visible from far away.
\end{eg}

%% file: tikz/dir.tex
\ctikzset{bipoles/length=.8cm}
\begin{tikzpicture}[xscale=1,yscale=1]
\def\c{-6}
\coordinate (a) at (-2.2+\c,0);
\coordinate (b) at (\c,1.45);
\coordinate (c) at (\c,-1.45);
\coordinate (d) at (\c+2.2,0);
\draw (a) -- (b) -- (d) -- (c) -- (a);
\draw (b) -- (c);
\draw[fill=white] (a) circle (3pt);
\draw[fill=black] (b) circle (3pt);
\draw[fill=black] (c) circle (3pt);
\draw[fill=white] (d) circle (3pt);

\def\a{5.5}
\draw[fill=black!20] ({-1+\a},-1) -- ({\a+1},-1) -- ({\a+1},1) -- ({\a-1},1) -- cycle;
\draw ({\a-1},-1.618) -- ({\a-1},1.618);
\draw ({\a+1},-1.618) -- ({\a+1},1.618);
\draw ({\a-1.618},-1) -- ({\a+1.618},-1);
\draw ({\a-1.618},1) -- ({\a+1.618},1);
\draw ({\a-1.618},-1.618) -- ({\a+1.618},1.618);

\def\b{0}
\coordinate (a1) at (-1.8+\b,-0.25);
\coordinate (b1) at (\b,0.95);
\coordinate (c1) at (\b,-1.45);
\coordinate (d1) at (1.8+\b,-0.25);
\draw (a1) to [R] (b1) to [R] (d1) to [R] (c1) to [R] (a1);
\draw (b1) to [R] (c1);
\draw (a1) to ++(-0.3,0) to ++(0,1.75) to[battery] ++(4.2,0) to ++(0,-1.75) to (d1);
\draw[fill=black] (a1) circle (2pt);
\draw[fill=black] (b1) circle (2pt);
\draw[fill=black] (c1) circle (2pt);
\draw[fill=black] (d1) circle (2pt);
\end{tikzpicture}

%% file: tikz/compjoin.tex
\begin{tikzpicture}[xscale=0.5,yscale=0.5]
\coordinate (a1) at (0,5);
\coordinate (a2) at ({5*sin(72)},{5*cos(72)});
\coordinate (a3) at ({5*sin(144)},{5*cos(144)});
\coordinate (a4) at ({5*sin(216)},{5*cos(216)});
\coordinate (a5) at ({5*sin(288)},{5*cos(288)});
\coordinate (b1) at (2,2.3);
\coordinate (b2) at (-2,2.3);
\coordinate (b3) at (-2,-1.7);
\coordinate (b4) at (2,-1.7);
\draw (a1) -- (b1);
\draw (a1) -- (b2);
\draw (a1) -- (b3);
\draw (a1) -- (b4);
\draw (a2) -- (b1);
\draw (a2) -- (b2);
\draw (a2) -- (b3);
\draw (a2) -- (b4);
\draw (a3) -- (b1);
\draw (a3) -- (b2);
\draw (a3) -- (b3);
\draw (a3) -- (b4);
\draw (a4) -- (b1);
\draw (a4) -- (b2);
\draw (a4) -- (b3);
\draw (a4) -- (b4);
\draw (a5) -- (b1);
\draw (a5) -- (b2);
\draw (a5) -- (b3);
\draw (a5) -- (b4);
\draw (b1) -- (b2);
\draw (b1) -- (b3);
\draw (b1) -- (b4);
\draw (b2) -- (b3);
\draw (b2) -- (b4);
\draw (b3) -- (b4);
\draw[fill=white] (a1) circle (5pt);
\draw[fill=white] (a2) circle (5pt);
\draw[fill=white] (a3) circle (5pt);
\draw[fill=white] (a4) circle (5pt);
\draw[fill=white] (a5) circle (5pt);
\draw[fill=black] (b1) circle (5pt);
\draw[fill=black] (b2) circle (5pt);
\draw[fill=black] (b3) circle (5pt);
\draw[fill=black] (b4) circle (5pt);
\end{tikzpicture}

%% file: sections/combo.tex
\section{Combinatorics of Dirichlet arrangements}
\label{sec:combo}

In this section we prove Theorems \ref{thm:maincomb} and \ref{thm:supintro}. Let $\g=(V,E)$ be a graph with boundary $\B\subsetneq V$ and boundary data $ u:\B\to \R$. Write $k=|E|$, $d=|V|$, $m=|\B|\geq 2$, and $n=|\inte|$. We denote by $\mg$ the graph obtained from $\g$ by adding an edge between each pair of boundary nodes.

Given an arrangement $\AA$, the \emph{intersection poset} of $\AA$ is the set $L(\AA)$ of nonempty intersections of elements of $\AA$, ordered by reverse inclusion and graded by codimension. Thus $X\leq Y$ in $L(\AA)$ means $X\supset Y$. If $\AA$ is central, then $L(\AA)$ is a geometric lattice. The \emph{characteristic polynomial} of $\AA$ is defined as the characteristic polynomial of $L(\AA)$ and is denoted by $\chi_\AA$.

The \emph{graphic arrangement} $\Ag$ of $\g$ is is the arrangement in $\R^d$ defined by
\[Q(\Ag)=\prod_{ij\in E} (x_i-x_j).\]
Graphic arrangements are well studied because one can translate between properties of $\Ag$ and corresponding properties of $\g$ \cite{doch2014, ehrenborg2009, imai1996, macinic2009, peeva2002, schenck2002}. The following theorem is the graphic version of Theorem \ref{thm:maincomb}. For proofs, see \cite{stanley2007}.

\begin{thm}
	For any graph $\g$, the following hold:
	\begin{enumerate}[(i)]
		\item The intersection poset $L(\Ag)$ is isomorphic to the lattice of connected partitions of $\g$
		\item The characteristic polynomial $\chi_{\Ag}$ is the chromatic polynomial of $\g$
		\item The chambers of $\Ag$ correspond to the acyclic orientations of $\g$.
	\end{enumerate}
	\label{thm:gra}
\end{thm}

\subsection{Intersection poset and connected partitions}
\label{sec:poset}

A \emph{connected partition} of $\g$ is a partition $\pi$ of $V$ whose blocks induce connected subgraphs of $\g$. The set $\cpt$ of connected partitions of $\g$ is a lattice ordered by refinement. That is, $X\leq Y$ in $\cpt$ means $X$ is a refinement of $Y$. A subset $S\subset \cpt$ is an \emph{order ideal} if for all $Y\in S$, $X\leq Y$ implies that $X\in S$.

\begin{mydef}
A connected partition of $\g$ is \emph{boundary-separating} if it belongs to
\[\seq=\{\pi\in \cpt:|P\cap\B|\leq 1\mbox{ for all }P\in \pi\}.\]
\end{mydef}

	\begin{proof}[Proof of Theorem \ref{thm:maincomb}(i)]
		Let $X\in L(\Au)$ and $x\in X$. For each $i\in V$ let $S_i\subset V$ be the set of $j\in V$ for which there exists a path $P$ from $i$ to $j$ such that $x_v$ is the same for all $v\in P$. We obtain an element $\lambda_X=\{S_i : i\in V\}$ of $\cpt$. No distinct boundary nodes $j$ and $j'$ can belong to a single block $S_i$, as this would imply that $ u(j) = u(j')$. Hence $\lambda_X\in\seq$.
		
		Now suppose that $\pi\in\seq$. We reverse the above construction. For every block $B\in\prt$, let $E_B\subset E$ be the subset of edges with both ends in $B$. These define an element
		\begin{equation*}
			Y_\pi = \bigcap_{B\in \prt} \bigcap_{e\in E_B} H_e
		\end{equation*}
		of $L(\Au)$, where each $H_e\in\Au$ is the hyperplane corresponding to $e$. It is not hard to see that $Y_{\lambda_{X}}=X$ and $\lambda_{Y_\pi}=\pi$. Moreover, for $X,X'\in L(\Au)$ we have $X\subset X'$ if and only if $\prt_{X'}\leq\prt_{X}$. The result follows.
	\end{proof}

\begin{cor}
	The intersection poset $L(\Au)$ depends only on $(\g,\B)$.
	\label{cor:poset}
\end{cor}

\begin{eg}
	Let $(\g,\B)$ be the Wheatstone bridge from Example \ref{eg:dir} with any boundary data $ u$. The Hasse diagram of $L(\Ag)$ is drawn in Figure \ref{fig:poset}, where $L(\Au)$ is the order ideal consisting of the darker vertices.  
	
	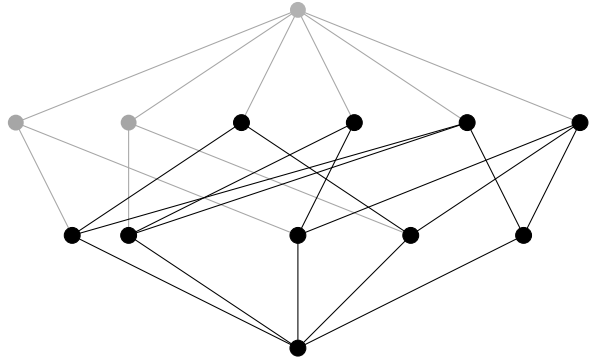
\begin{figure}[ht]
		\centering
		\input{tikz/poset}
		\caption{The Hasse diagram of $\cpt$, where $(\g,\B)$ is the Wheatstone bridge, with the Hasse diagram of $\seq$ marked in black.}
		\label{fig:poset}
	\end{figure}
	
	\label{eg:poset}
\end{eg}

\subsection{Characteristic polynomial and precolorings}

We now prove Propositions \ref{prop:ffm} and \ref{prop:pcpfor} below, which together imply Theorem \ref{thm:maincomb}(ii). We then discuss log-concavity of the coefficients of $\chi_{\Au}$.

For positive integers $\l$, write $[\l]=\{1,\ldots,\l\}$. Recall that a (proper) \emph{$\l$-coloring} of $\g$ is a function $V\to [\l]$ taking distinct values on adjacent vertices. Also recall that the \emph{chromatic polynomial} $\cp$ of $\g$ is a polynomial with integer coefficients such that $\cp(\l)$ is the number of $\l$-colorings of $\g$ for all integers $\l\geq 1$.

Let $c:\B\to[m]$ be a bijection. Herzberg and Murty \cite{herzberg2007} exhibited a polynomial $\pcp$ with integer coefficients such that
\[\pcp(\l)=|\{\widehat{c}:V\to [p] \mid \widehat{c}
\mbox{ is an }\l\mbox{-coloring of }\g
\mbox{ that extends }c\}|\]
for all integers $\l\geq m$. The polynomial $\pcp$ is fundamental in the study of Sudoku puzzles \cite{herzberg2007} and in the \emph{Precoloring Extension Problem} \cite{biro1992, chartrand2008}.

\begin{mydef}
	We call $\pcp$ the \emph{precoloring polynomial} of $(\g,\B)$.
\end{mydef}

The following result is due implicitly to Crapo and Rota \cite[Section 17]{crapo1970} and was isolated later by Athanasiadis \cite{athanasiadis1996}. The resulting \emph{Finite Field Method} is a powerful means of computing characteristic polynomials of arrangements.

\begin{prop}[{\cite[Theorem 2.2]{athanasiadis1996}}]
	Suppose that $\AA$ is an arrangement in $\R^d$ defined over $\Z$. Fix a prime $p\in\Z$, and let $\AA^p$ be the arrangement in $\FF_p^d$ obtained by reducing the defining equations of $\AA$ mod $p$. If $p$ is sufficiently large, then $\chi_{\AA}(p)=|M(\AA^p)|$.
	
	\label{prop:ffm}
\end{prop}

\begin{prop}
	The characteristic polynomial of $\Au$ is the precoloring polynomial $\pcp$.
	
	\begin{proof}
		Fix a bijection $c:\B\to [m]$, and set boundary data $u=c$. Corollary \ref{cor:poset} implies that $\chi_{\Au}$ is unaffected by the choice of $u$. Consider $\FF_p^n$ as the set $[p]^n$. We can assign to any point $x\in M(\Au^p)$ an element of
		\begin{equation*}
			\{\widehat{c}:V\to [p] \mid \widehat{c}
			\mbox{ is a }p\mbox{-coloring of }\g
			\mbox{ that extends }c\}
		\end{equation*}
		by setting $\hat{c}(i)=x_i$ for all $i\in \inte$. This assignment is easily seen to be a bijection, whence $\pcp(p)=|M(\Au^p)|$. The result now follows from Proposition \ref{prop:ffm} and the fact that $\pcp$ is a polynomial, since $\pcp(p)=|M(\Au^p)|$ for infinitely many $p$.
	\end{proof}
\label{prop:cp}
\end{prop}

Since $\pcp$ is the characteristic polynomial of an arrangement, it satisfies a deletion-restriction formula and a broken circuit theorem (see \cite{stanley2007}). There is also a corresponding Tutte polynomial (see \cite{ardila2007}). We will not explore these ideas here.

\begin{eg}
		Let $(\g,\B)=\comp{m}{n}$ as in Example \ref{eg:compjoin}. Fix a bijection $c:\B\to [m]$ and an integer $\l\geq d$. To extend $c$ to an $\l$-coloring of $\g$, we must choose for every interior vertex a color that has not yet been used. This accounts for $(\l-m)!/(\l-d)!$ possible extensions of $c$, and there are no others. Hence
	\begin{equation*}
		\pcp(t)= (t-m)_n = (t-m)(t-m-1)\cdots (t-d+1)
	\end{equation*}
	is a falling factorial.
	\label{eg:comppoly}
\end{eg}

In Example \ref{eg:comppoly} the precoloring polynomial $\pcp(t)$ divides the chromatic polynomial $\chi_{K_n}(t)$, where $K_n$ is the complete graph on $n$ vertices. This is a consequence of the following proposition.

\begin{prop}
	The precoloring polynomial $\pcp$ satisfies
	\[\chi_{\mg}(t)=(t)_m\cdot \pcp(t),\]
	where $\cdot$ denotes multiplication and $(t)_m=t(t-1)(t-2)\cdots (t-m+1)$ denotes a falling factorial.
	\begin{proof}
		Fix $\l\geq d$. We count the number of $\l$-colorings of $\widehat{G}$ vertex-by-vertex, starting with the boundary nodes. Since $\B$ is a clique in $\widehat{G}$, there are $\l$ ways to color the first boundary node, $\l-1$ ways to color the second, and $\l-r+1$ ways to color the $r$th. Once all the boundary nodes are colored, the number of ways to color the interior vertices is $\pcp(\l)$. Thus $\chi_{\mg}(t)=(t)_m\cdot \pcp(t)$ holds for infinitely many $t$, so it holds in general.
	\end{proof}
\label{prop:pcpfor}
\end{prop}

\begin{eg}
	Let $\g$ be the path graph on $d\geq 3$ vertices, and let $\B$ consist of both ends of the path. We have $\mg=C_d$, the cycle graph on $d$ vertices. Using the elementary formula $\chi_{C_d}(t)=(t-1)^d+(-1)^d(t-1)$, we obtain
	\begin{align*}
		\pcp(t)&=\frac{\chi_{C_d}(t)}{t(t-1)}\\
		&= \prod_{r=1}^{n} (t+\zeta^r-1),
	\end{align*}
	where $\zeta\in \C$ is any primitive $k$th root of unity. 
	\label{eg:pathcp}
\end{eg}

%We remark briefly on log-concavity. Write the chromatic polynomial $\cp$ of $\g$ momentarily as
%\[\cp(t)=a_dt^n - a_{d-1}t^{d-1} + \cdots + (-1)^d a_0.\]
%Huh \cite{huh2012} showed that the coefficients $a_r$ form a \emph{log-concave} sequence: i.e., that
%\[a_r^2\geq a_{r-1}a_{r+1}\]
%for all $r\in [d-1]$. This result confirmed a long-standing conjecture of Read \cite[\S 9]{read1968} that the sequence $a_r$ is \emph{unimodal}: i.e., that there exists $s\in %\{0,1,\ldots,d\}$ such that
%\[a_0\leq a_1 \leq \cdots \leq a_{s-1}\leq a_s \geq a_{s+1} \geq \cdots \geq a_d.\]
%In fact, Huh proved log-concavity of the sequence of coefficients of the characteristic polynomial of any real hyperplane arrangement. Combining this fact with Proposition %\ref{prop:pcpfor} yields the following corollary.

A sequence $a_0,\ldots,a_n$ of positive numbers is \emph{log-concave} if
\[a_r^2\geq a_{r-1}a_{r+1}\]
for all $r\in [n-1]$. A log-concave sequence is necessarily \emph{unimodal}; i.e., there exists $s\in \{0,1,\ldots,n\}$ such that
\[a_0\leq a_1 \leq \cdots \leq a_{s-1}\leq a_s \geq a_{s+1} \geq \cdots \geq a_n.\]

\begin{cor}
	Suppose that $\g$ contains a clique on $\ell$ vertices, and write
	\[\cp(t)/(t)_\ell = a_0t^n - a_1t^{n-1} + \cdots + (-1)^n a_n.\]
	The sequence $a_0,\ldots,a_n$ is log-concave.
	\label{cor:chrom}
\begin{proof}
	This is an application of results of Huh \cite{huh2012} to Proposition \ref{prop:pcpfor}.
\end{proof}
\end{cor}

\begin{remark}
	Corollary \ref{cor:chrom} does not hold for general polynomials divisible by $(t)_\ell$. That is, if $f(t)$ is a polynomial divisible by $(t)_\ell$ in $\Z[t]$ whose coefficients form a log-concave sequence, then the coefficients of $f(t)/(t)_\ell$ do not necessarily form a log-concave sequence, even if we require that $f(t)/(t)_\ell$ is monic with coefficients that alternate in sign. Take, for example, the polynomial
\begin{equation}
(t)_3\cdot (t^2-t+2) = t^5 -4t^4 +7t^3 -8t^3 +4t.
\label{eq:nochrom}
\end{equation}
Corollary \ref{cor:chrom} says that \eqref{eq:nochrom} is not the chromatic polynomial of a graph containing a 3-cycle, since the coefficients of $t^2 - t + 2$ do not form a log-concave sequence.
\end{remark}

\subsection{Chambers and compatible orientations}
\label{sec:chambers}

We prove Theorem \ref{thm:chamberbij} below, which implies Theorem \ref{thm:maincomb}(iii). We then give formulas for the number of chambers and bounded chambers of $\Au$.

Given a real arrangement $\AA$, we denote by $\CC(\AA)$ and $\BC(\AA)$ the sets of chambers and bounded chambers, resp., of $\AA$. There is a bijection between the chambers of $\CC(\Ag)$ and set of the acyclic orientations of $\g$ due to Greene \cite{greene1976}. Namely, to any $C\in \CC(\Ag)$ we take $x\in C$ and assign the orientation $\oo(C)$ of $\g$ with $\vec{ij}$ if and only if $x_i>x_j$ for all $ij\in E$.

	We say that an orientation $\sigma$ of $\g$ \emph{respects $ u$} if for any path $i\to j$ in $\sigma$ between boundary nodes $i$ and $j$ we have $ u(i)> u(j)$. Denote by $\comh$ the set of acyclic orientations of $\g$ that respect $u$. Let $\com\subset\comh$ be the subset of those orientations with no sinks or sources in $\inte$.
	
	\begin{mydef}
		The orientations in $\comh$ and $\com$ are called \emph{semicompatible} and \emph{compatible}, resp.
	\end{mydef}
	
	Consider the edges of $\g$ as resistors with arbitrary conductances $\con\in (0,\infty)^k$. Current flows from vertices of higher voltage to vertices of lower voltage. As $\con$ varies, the compatible orientations are the orientations of all current flows through $\g$ that respect the boundary voltages $ u$ and in which the current across every edge is nonzero. The next proposition, which generalizes \cite[Theorem a.1]{lempel1967}, reinforces this point of view.

\begin{prop}
	The following are equivalent:
	\begin{enumerate}[(i)]
		\item $\mg$ is 2-connected
		\item $(\g, u)$ admits a compatible orientation for any boundary data $ u$
		\item Every interior vertex of $\g$ lies on a simple undirected path in $\g$ between distinct boundary nodes.
	\end{enumerate}
	\begin{proof}
		We prove the equivalence of (i) and (iii). The equivalence of (i) and (ii) will follow from Theorem \ref{thm:chambercount} below.
		
		Suppose that (i) holds. Let $i\in\inte$. If there is no simple path in $\g$ connecting $i$ to $\B$, then $\mg$ is disconnected, a contradiction. Suppose instead that there is a simple path in $\g$ connecting $i$ to a boundary node $j$, but that there is no simple path containing $i$ and two distinct boundary nodes. Notice that $\mg\setminus j$ is disconnected, so $\mg$ is not 2-connected, a contradiction. Hence (iii) holds.
		
		Now suppose that (i) does not hold. Let $i\in V$ be such that $\mg\setminus i$ is disconnected. Since $\B$ forms a clique in $\mg$, all boundary nodes remaining in $\mg\setminus i$ belong to the same component $X$ of $\mg\setminus i$. Let $j$ be a vertex of $\mg\setminus i$ not in $X$. Any path in $\g$ that contains $j$ and begins and ends at distinct boundary nodes $j$ must contain at least 2 edges (with multiplicity) incident to $i$. Such a path is not simple, so (iii) does not hold.
	\end{proof}
\label{prop:3things}
\end{prop}

\begin{thm}
	There is a bijection from the set of chambers (resp., bounded chambers) of $\Au$ to the set of semicompatible (resp., compatible) orientations of $(\g, u)$.
	
	\begin{proof}
		First we show that $\oo$ is a bijection $\CC(\Au)\to \comh$. Suppose that $C\in \CC(\Au)$, and let $x\in C$. Since $x_i= u(i)$ for all $i\in \B$, $\oo(C)$ respects $ u$. Since $M(\Au)\subset M(\Ag)$, $\oo(C)$ is acyclic. Hence $\oo(C)\in \comh$. Clearly $\oo$ is injective.
		
		Now suppose that $\sigma\in \comh$, and note that $ u$ defines a total order on $\B$. Since $\sigma$ is acyclic, we obtain a partial order on $V$ by setting $j\leq i$ if and only if $\vec{ij}\in\sigma$. Extend this order to a total order on $V$; such an extension also extends the total order on $\B$. Thus we can take $y\in M(\Au)$ whose entries respect the total order on $V$. Write $\oi(\sigma)$ for the chamber of $\Au$ containing $y$. We have $\oo(\oi(\sigma))=\sigma$, so $\oo$ is a bijection $\CC(\Au)\to \comh$, as desired.
		
		We must now show that $\sigma\in \com$ if and only if $\oi(\sigma)\in \BC(\Au)$. For the ``if'' direction, suppose that $\sigma\in \comh\setminus \com$, and suppose without loss of generality that $i\in \inte$ is a source of $\sigma$. Let $x\in \oi(\sigma)$, and let $y\in\R^n$ be the standard basis element corresponding to $i$. Let $t>0$ be large enough that $x+ty\in M(\Au)$, and let $C\in \CC(\Au)$ be the chamber containing $x+ty$. Clearly $i$ is a source of $\oo(C)$, and in fact $\sigma=\oo(C)$. Hence $\oi(\sigma)$ is unbounded, proving the ``if'' direction.
		
		For the ``only if'' direction, suppose that $\sigma\in \com$. Let $f\in \oi(\sigma)$, and let $\XX$ be as in Definition \ref{def:Au}. We show that any ray in $\XX$ originating at $f$ is not contained in the convex set $\oi(\sigma)$. Let $g\in \R^n\setminus\{0\}$ with $g_i=0$ for all $i\in \B$, and suppose without loss of generality that $g_v>0$ for some $v\in \inte$. For large enough $t>0$ we have $f+tg\in M(\Au)$ and $f_v+tg_v> u(w)$ for all $w\in \B$. If $C\in \CC(\Au)$ is the chamber containing $f+tg$, then $\oo(C)$ has a source in $\inte$. Hence $C\neq \oi(\sigma)$. Since the direction of the ray in $\XX$ was arbitrary, we conclude that $\oi(\sigma)$ is bounded.
	\end{proof}
	\label{thm:chamberbij}
\end{thm}

Zaslavsky \cite{zaslavsky1975} expressed the numbers of chambers and bounded chambers of a real arrangement $\AA$ in terms of the characteristic polynomial $\chi_\AA$. We are particularly interested in counting the bounded chambers of $\Au$ because of their later role in Section \ref{sec:mfu}.

\begin{prop}[{\cite[Theorems A and C]{zaslavsky1975}}]
	If $\AA$ is a real arrangement, then the number of chambers of $\AA$ is $|\chi_{\AA}(-1)|$, and the number of bounded chambers is $|\chi_{\AA}(1)|$.
	\label{prop:zas}
\end{prop}

 Proposition \ref{prop:zas} gives $|\CC(\Au)|$ and $|\BC(\Au)|$ in terms of the precoloring polynomial $\pcp$. The next theorem gives these counts in terms of a genuine chromatic polynomial. We let
 \[\beta(\g)=|\cp'(1)|,\]
 where $\cp'$ is the derivative of $\cp$. The integer $\beta(\g)$ is called the \emph{beta invariant} of $\g$ \cite{benashski1995,oxley1982}. We have $\beta(\g)>0$ if and only if $\g$ is 2-connected.

\begin{thm}
	The number of semicompatible orientations of $(\g, u)$ is
	\begin{equation}
		|\comh|=\frac{\alpha(\mg)}{m!},
		\label{eq:comhfor}
	\end{equation}
	where $\alpha(\mg)$ is the number of acyclic orientations of $\mg$. The number of compatible orientations is
	\begin{equation}
		|\com|=\frac{\beta(\mg)}{(m-2)!},
		\label{eq:comfor}
	\end{equation}
	where $\beta(\mg)$ is the beta invariant of $\mg$.
	\label{thm:chambercount}
	
	\begin{proof}
		Proposition \ref{prop:pcpfor} says that
		\begin{equation}
		\chi_{\mg}(t)=(t)_m\cdot\pcp(t).
		\label{eq:pcpfor2}
		\end{equation}
		Evaluating both sides of \eqref{eq:pcpfor2} at $t=-1$ and rearranging gives $|\pcp(-1)|=|\chi_{\mg}(-1)|/m!$. Proposition \ref{prop:zas} implies that $|\chi_{\mg}(-1)|=\alpha(\mg)$. Now \eqref{eq:comhfor} follows from Theorem \ref{thm:chamberbij}.
		
		Taking derivatives of both sides of \eqref{eq:pcpfor2}, evaluating at $t=1$ and rearranging, we have $|\pcp(1)|=|\chi_{\mg}'(1)|/(m-2)!$. Thus \eqref{eq:comfor} follows from Proposition \ref{prop:zas} and Theorem \ref{thm:chamberbij}.
	\end{proof}
\end{thm}

The count $|\comh|=|\pcp(-1)|$ was obtained by Jochemko and Sanyal \cite[Corollary 4.5]{jochemko2014}, who used a combinatorial reciprocity for $\pcp$. Equation \eqref{eq:comfor} seems to be the first analogous treatment of $\pcp(1)$.

\begin{eg}
	Suppose that $\B=\{i,j\}$ with any boundary data $ u$. Here the orientations in $\com$ are called \emph{$ij$-bipolar} and have applications to graph drawing \cite{fraysseix1995}. If one considers the edges of $\g$ as resistors with arbitrary positive conductances, then the $ij$-bipolar orientations of $\g$ are the possible orientations of current flow through $\g$ in which the current flowing through each resistor is nonzero after a battery is put across $i$ and $j$. In this case, the formula \eqref{eq:comfor} was observed by Abrams and Kenyon \cite{abrams2016}.  
\end{eg}

\begin{proof}[Proof of Corollary \ref{cor:vis}]
	This follows from Example \ref{eg:vis} and Theorem \ref{thm:chambercount}.
\end{proof}

\begin{q}
	Let $\Gamma$ be the graph with vertex set $\BC(\Au)$ and an edge between chambers $C$ and $C'$ whenever $C$ and $C'$ are adjacent in $\R^n$. O.~de Mendez showed in \cite{demendez1994} that if $m=2$ and $\mg$ is 3-connected, then $\Gamma$ is connected (see \cite[Theorem 7.1]{fraysseix1995}). Is $\Gamma$ connected whenever $\mg$ is 3-connected?
\end{q}

%% file: tikz/poset.tex
\begin{tikzpicture}[scale=1.5]
\coordinate (a) at (0,0);
\coordinate (b1) at (-2.5,1);
\coordinate (b2) at (-2,1);
\coordinate (b3) at (-1.5,1);
\coordinate (b4) at (0,1);
\coordinate (b5) at (1,1);
\coordinate (b6) at (2,1);
\coordinate (c1) at (-2.5,2);
\coordinate (c2) at (-1.5,2);
\coordinate (c3) at (-1,2);
\coordinate (c4) at (-0.5,2);
\coordinate (c5) at (0.5,2);
\coordinate (c6) at (1.5,2);
\coordinate (c7) at (2.5,2);
\coordinate (d) at (0,3);
\draw[black!35] (c1) -- (d);
\draw[black!35] (c2) -- (d);
\draw[black!35] (c1) -- (b2);
\draw[black!35] (c1) -- (b4);
\draw[black!35] (c2) -- (b3);
\draw[black!35] (c2) -- (b5);
\draw[black!35] (c4) -- (d);
\draw[black!35] (c5) -- (d);
\draw[black!35] (c6) -- (d);
\draw[black!35] (c7) -- (d);
\draw[black] (a) -- (b2);
\draw[black] (a) -- (b3);
\draw[black] (a) -- (b4);
\draw[black] (a) -- (b5);
\draw[black] (a) -- (b6);
\draw[black] (c4) -- (b2);
\draw[black] (c4) -- (b5);
\draw[black] (c5) -- (b3);
\draw[black] (c5) -- (b4);
\draw[black] (c6) -- (b2);
\draw[black] (c6) -- (b3);
\draw[black] (c6) -- (b6);
\draw[black] (c7) -- (b4);
\draw[black] (c7) -- (b5);
\draw[black] (c7) -- (b6);

\draw[fill=black] (a) circle (2pt);
\draw[fill=black] (b2) circle (2pt);
\draw[fill=black] (b3) circle (2pt);
\draw[fill=black] (b4) circle (2pt);
\draw[fill=black] (b5) circle (2pt);
\draw[fill=black] (b6) circle (2pt);
\fill[black!35] (c1) circle (2pt);
\fill[black!35] (c2) circle (2pt);
\draw[fill=black] (c4) circle (2pt);
\draw[fill=black] (c5) circle (2pt);
\draw[fill=black] (c6) circle (2pt);
\draw[fill=black] (c7) circle (2pt);
\fill[black!30] (d) circle (2pt);
\end{tikzpicture}

%% file: sections/super.tex
\section[Supersolvability and psi-graphical arrangements]{Supersolvability and \texorpdfstring{$\psi$}{psi}-graphical arrangements}
\label{sec:super}

We prove Theorem \ref{thm:stansup} below, building on results of \cite{mu2015,stanley2015,suyama2018}. This will imply Theorem \ref{thm:supintro}. Stanley \cite{stanley2015} introduced the following class of arrangements to study visibility arrangements of order polytopes (see Example \ref{eg:vis}).

\begin{mydef}
	Denote the power set of $\R$ by $\mathcal{P}(\R)$, and let $\psi:V\to \mathcal{P}(\R)$ be such that $|\psi(i)|<\infty$ for all $i\in V$. Let $\Ap$ be the arrangement in $\R^n$ of hyperplanes $\{x_i=x_j\}$ for all $ij\in E$ and $\{x_i=\alpha\}$ for all $i\in V$ and $\alpha\in \psi(i)$. An arrangement is called \emph{$\psi$-graphical} if it is of the form $\Ap$ for some pair $(\g,\psi)$.  
	\label{def:Ap}
\end{mydef}

It turns out that every Dirichlet arrangement can be realized as a $\psi$-graphical arrangement, and vice versa. We prove this equivalence. The main benefit of our definition over Definition \ref{def:Ap} is that it renders more natural and intuitive descriptions of combinatorial features of Dirichlet arrangements that closely resemble their graphic counterparts. Theorem \ref{thm:supintro} and the results of Section \ref{sec:combo} are just a few examples of this.

\begin{prop}
	The classes of Dirichlet arrangements and $\psi$-graphical arrangements are equal.
	
	\begin{proof}
		Let $\g=(V,E)$ be a graph with boundary $\B$ and boundary data $ u$. Let $\gint$ be the subgraph of $\g$ induced by $\inte$, and let $\pint:\inte\to \mathcal{P}(\R)$ be given by $\pint(i)=\{ u(j) : j\sim i\mbox{ and }j\in\B\}$. We have $\Au=\AA_{\gint,\pint}$. Hence every Dirichlet arrangement is $\psi$-graphical.
		
		Now consider a $\psi$-graphical arrangement $\Ap$, and let $S=\bigcup_{i\in V} \psi(i)$. Let $V'=V\cup \{j_s : s\in S\}$ and $E'=\{ij_s : i\in V \mbox{ and } s\in \psi(i)\}$. Also let $\g'=(V',E')$, and let $ u' : \{j_s : s\in S\}\to \R$ be given by $u(j_s) = s$ for all $s\in S$. It is not hard to see that $\Ap=\AA_{\g', u'}$. Hence every $\psi$-graphical arrangement is Dirichlet.
	\end{proof}
\label{prop:psigra}
\end{prop}

Until now, research on $\psi$-graphical arrangements has focused on questions of supersolvability and freeness. Supersolvable arrangements enjoy a number of useful combinatorial, topological and algebraic properties, and are fundamental in the study of hyperplane arrangements. We do not discuss \emph{freeness} of arrangements in this paper, except to remark that it is a far-reaching generalization of supersolvability. For more on supersolvability and freeness, see \cite{orlik1992}.
 
 Let us recall the definition supersolvability for non-central arrangements. An element $X$ of a lattice $L$ is called \emph{modular} if
 \[\rk(X)+\rk(Y)=\rk(X\wedge Y)+\rk(X\vee Y)\]
 for all $Y\in L$, where $\wedge$ denotes the meet in $L$ and $\vee$ denotes the join.
 
 \begin{mydef}
 	A non-central arrangement $\AA$ is \emph{supersolvable} if there exists a maximal chain of $L(c\AA)$ consisting entirely of modular elements.
 \end{mydef}
 
  A \emph{perfect elimination ordering} of $\g$ is an ordering $i_1,\ldots,i_n$ of the vertices such that every $i_s$ is simplicial in the subgraph induced by $\{i_s,\ldots,i_n\}$. If a perfect elimination ordering of $\g$ exists, then $\g$ is called \emph{chordal}. The following proposition summarizes results of Edelman and Stanley \cite{edelman1994,stanley1972}.
  
  \begin{prop}[{\cite[Theorem 3.3]{edelman1994} and \cite[Proposition 2.8]{stanley1972}}]
  	The following are equivalent:
  	\begin{enumerate}[(i)]
  		\item $\g$ is chordal
  		\item $\Ag$ is supersolvable
  		\item $\Ag$ is free.
  	\end{enumerate}
  \label{prop:supersolvable}
  \end{prop}
  
 %Let $\mg$ be the graph obtained from $\g$ by adding an edge between each pair of boundary nodes, and suppose that $\mg$ is chordal. Since $\B$ is a clique in $\mg$, a result of Rose \cite[p.~603]{rose1970} says that there is a perfect elimination ordering $i_1,\ldots,i_n$ of $\mg$ whose last $m$ vertices are in $\B$. A naive counting argument gives
  %\begin{equation}
  %	\pcp(t)=\prod_{r=1}^{n-m} (t-d_r),
  %	\label{eq:pcrfac}
  %\end{equation}
  %where $d_r=|\{s>r : i_r\sim i_s\}|$. Stanley showed that if an arrangement $\AA$ is supersolvable, then the roots of $\chi_\AA$ are nonegative integers \cite{stanley1972}. Hence \eqref{eq:pcrfac} and Proposition \ref{prop:cp} suggest the possibility that $\Au$ is supersolvable when $\mg$ is chordal.
  
  We provide a direct analog for Dirichlet arrangements. Our characterization is based on work of Mu--Stanley and Suyama--Tsujie \cite{mu2015,suyama2018}.
  
  \begin{thm}[{Generalization of Theorem \ref{thm:gra}(iv)}]
  	The following are equivalent:
  	\begin{enumerate}[(i)]
  		\item $\mg$ is chordal
  		\item $\Au$ is supersolvable
  		\item $\Au$ is free.
  	\end{enumerate}
  	\begin{proof}
  		Let $\gint$ denote the subgraph of $\g$ induced by $\inte$, and let $\pint:\inte\to 2^\R$ be given by
  		\[\pint(i)=\{ u(j) : j\sim i\mbox{ and }j\in\B\}.\]
  		Notice that $\Au=\AA_{\gint,\pint}$. A \emph{weighted elimination ordering} of $(\gint,\pint)$ is a perfect elimination ordering $i_1,\ldots,i_{n-m}$ of $\gint$ such that if $i_r\sim i_s$ with $r<s$, then $\pint(i_r)\subset \pint(i_s)$. We show that $(\gint,\pint)$ admits a weighted elimination ordering if and only if $\mg$ is chordal. Theorem \ref{thm:stansup} will then follow from \cite[Theorem 2.2]{suyama2018}.
  		
  		Suppose that $i_1,\ldots,i_n$ is a weighted elimination ordering of $(\gint,\pint)$. We claim that
  		\[i_1,\ldots,i_n,j_1,\ldots,j_m\]
  		is a perfect elimination ordering of $\mg$ for any ordering $j_1,\ldots,j_m$ of $\B$. Suppose that $i_r\sim i_s$ and $i_r\sim i_t$ for $r<s,t$. Clearly the same adjacencies hold in $\mg$. Now suppose that $i_r\sim i_s$ and $i_r\sim j$ for some $j\in\B$. Since $r<s$ we have $\pu(i_r)\subset \pu(i_s)$, so $u(j)\in \pu(i_s)$. Hence $i_s\sim j$ in $\mg$. Now suppose without loss of generality that $i_r\sim j_1$ and $i_r\sim j_2$. Since $\B$ is a clique in $\mg$, we have $j_1\sim j_2$. The claim follows, proving that $\mg$ is chordal.
  		
  		Conversely, suppose that $\mg$ is chordal. Since $\B$ is a clique, there is a perfect elimination ordering of $\mg$ whose last $m$ vertices are the elements of $\B$ by a result of Rose \cite[p.~603]{rose1970}. The first $n$ vertices of this perfect elimination ordering form a weighted elimination ordering of $(\gint,\pint)$. 
  	\end{proof}
    	\label{thm:stansup}
  \end{thm}

\begin{eg}
	Let $(\g,\B)=\comp{m}{n}$ with any boundary data $u$. Since $\mg$ is complete, any ordering of $V$ is a perfect elimination ordering of $\mg$. Hence $\Au$ is supersolvable.
	
	More generally, suppose that $\inte$ is a clique in $\g$, but make no assumptions about which boundary nodes and which interior vertices are adjacent. Write $\B=\{j_1,\ldots,j_m\}$ and $\inte=\{i_1,\ldots,i_{n-m}\}$ so that if $i_r$ is adjacent to a boundary node and $r<s$, then $i_s$ is adjacent to a boundary node. Notice that $i_1,\ldots,i_{n-m},j_1,\ldots,j_m$ is a perfect elimination ordering of $\mg$. Hence $\Au$ is supersolvable for any boundary data $ u$. 
\end{eg}

\begin{eg}
	In this example $\Ag$ is supersolvable but $\Au$ is not. Let $\g$ be a path graph on $n\geq 3$ vertices with $\B$ consisting of both ends of the path. In Example \ref{eg:pathcp} we computed
	\[\pcp(t)=\prod_{r=1}^{k-1} (t+\zeta^r-1),\]
	where $\zeta\in \C$ is a primitive $k$th root of unity. At most one root of $\pcp$ is a positive integer. Hence when $n\geq 4$, $\Au$ is not supersolvable for any boundary data $ u$. Alternatively, it is easy to see that no vertex in the cycle graph $\mg$ is simplicial.  
\end{eg}

%% file: sections/master.tex
\section{Master functions and electrical networks}
\label{sec:master}

 Given an arrangement $\AA$ in $\R^d$, let $\AA_\C$ be the arrangement in $\C^d$ defined by $Q(\AA_\C)=Q(\AA)$, where $Q(\AA_\C)$ is considered as a polynomial over $\C$. In other words, $\AA_\C=\AA\otimes_\R\C$ is the complexification of $\AA$. We think of $M(\AA)=M(\AA_\C)\cap \R^d$ as the set of real points of $M(\AA_\C)$.

\begin{mydef}
	Let $\AA$ be an arrangement of $k$ hyperplanes in $\R^d$. The \emph{master function} of $\AA$ with weights $\eng\in \C^k$ is the multivalued function $\mf:M(\AA_\C)\to \C$ given by
\begin{equation}
\mf(x) = \sum_{r=1}^k \eng_r\log f_r(x),
\end{equation}
where the $f_r$ are the defining functions of $\AA$.
\end{mydef}

\begin{mydef}
	A point $x\in M(\AA_\C)$ is a \emph{critical point} of the master function $\mf$ if $\grad{\mf}{x} = 0$. That is,
\begin{equation}
\sum_{r=1}^k \frac{\partial f_r}{\partial x_i} \frac{\eng_r}{f_r(x)}=0
\label{eq:cpeq}
\end{equation}
for all $i=1,\ldots,d$.
\label{def:crit}
\end{mydef}

Definition \ref{def:crit} makes sense because the difference of any two branches of $\mf$ is a constant function. Moreover the critical points of $\mf$ are independent of the choices of $f_r$. As $i$ ranges over $1,\ldots,d$, the equations \eqref{eq:cpeq} are sometimes called the \emph{Bethe Ansatz equations} for $\mf$ (see \cite[Section 12.1]{sottile2011}). We denote the set of critical points of $\mf$ by $\VV$.

The term \emph{master function} sometimes refers to the product $x\mapsto \exp(\mf(x))$ of powers of affine functionals. Proposition \ref{prop:varchenko} below is due to Varchenko \cite{varchenko1995} and is foundational in the study of master functions. Given $S\subset \C^d$ and a list of mutually disjoint sets $A_1,\ldots,A_\ell\subset \C^d$, we say that the elements of $S$ form a \emph{system of distinct representatives} for the sets $A_1,\ldots,A_\ell$ if $|S|=\ell$ and $S\cap A_r$ is nonempty for all $r=1,\ldots,\ell$.

\begin{prop}[{\cite[Theorem 1.2.1]{varchenko1995}}]
	Let $\AA$ be a real essential arrangement and $\eng\in (0,\infty)^{\AA}$. The critical points of the master function $\mf$ form a system of distinct representatives for the bounded chambers of $\AA$.
	\label{prop:varchenko}
\end{prop}

For a short, elementary proof of Proposition \ref{prop:varchenko}, see \cite[\S9.2]{sottile2011}.

\subsection{Laplacians and master functions}

If every hyperplane in $\AA$ contains the origin, then the defining functions $f_r$ of $\AA$ are homogeneous. In this case we let $L=L(\con)$ be the $d\times d$ matrix in the usual basis of $\R^d$ with
\begin{equation}
x^T Lx=\sum_{r=1}^k \con_r f_r(x)^2
\label{eq:quadform}
\end{equation}
for all $x\in \R^d$, where $x^T$ is the transpose of $x$. We call $L$ the \emph{Laplacian matrix} of $\AA$ with weights $\con$. Our terminology is explained by Example \ref{eg:gralap} below, which features in the remainder of this section.

\begin{eg}
	We let $\lg=\lg(\con)$ denote the Laplacian matrix of the graphic arrangement $\Ag$ with weights $\con$. Here $\lg$ is just the weighted Laplacian matrix of $\g$, where each edge $e$ is weighted by $\con_e$. Entrywise, we have
	\begin{equation*}
		[\lg]_{ij} =
		\begin{cases}
			\sum_{v\sim i} \con_{iv} &\mbox{if }i=j\\ -\con_{ij}&\mbox{if }i\sim j\\ 0&\mbox{else}.
		\end{cases}
	\end{equation*}
	The quadratic form associated with $\lg$ is
	\begin{equation*}
		x^T \lg x=\sum_{ij\in E} \con_{ij} (x_i-x_j)^2
	\end{equation*}
	for all $x\in \R^d$. If $\g$ as an electrical network with conductances $\con\in (0,\infty)^k$, and voltages $x$, then $x^T\lg x$ is the total energy dissipated by the network.  
	
	\label{eg:gralap}
\end{eg}

One can think of $\mf$ as a (weighted) \emph{logarithmic barrier function}. Hessian matrices of logarithmic barrier functions play an important role in interior point methods (see, e.g., \cite{nesterov1994}). The next proposition connects Laplacian matrices of an arrangement $\AA$ to gradients and Hessian matrices of master functions of $\AA$. Let $\ef:\C^k\times \C^d\to \C^k$ be given by
\[\ef(\con,x)=(\con_1 f_1(x)^2,\ldots,\con_k f_k(x)^2).\]
We write $\Psi = \ef$. For suitable functions $g$, we let $\hess{g}{x}$ denote the Hessian matrix of $g$, evaluated at $x$. 

\begin{prop}
	If every hyperplane in $\AA$ contains the origin and $\eng=\Psi(\con,x)$ for some $x\in M(\AA_\C)$, then $\grad{\mf}{x}=Lx$ and $\hess{\mf}{x}=-L$.
	\begin{proof}
		First, notice that
		\begin{equation}
		L_{ij} = \sum_{r=1}^k \frac{\partial f_r}{\partial x_i} \frac{\partial f_r}{\partial x_j}\con_r.
		\label{eq:lapdiff}
		\end{equation}
		If $\eng=\Psi(\con,x)$ for some $x\in M(\AA_\C)$, then
		\begin{align*}
		\frac{\partial}{\partial x_i} \mf(x)
		&= \sum_{r=1}^k \frac{\partial f_r}{\partial x_i} \frac{\eng_r}{f_r(x)}\\
		&=\sum_{r=1}^k \frac{\partial f_r}{\partial x_i} \con_rf_r(x)\\
		&=\sum_{r=1}^k  \frac{\partial f_\ell}{\partial x_i} \con_r\left(\sum_{j=1}^d \frac{\partial f_r}{\partial x_j}x_j\right) \\
		&=\sum_{j=1}^d \left(\sum_{r=1}^k \frac{\partial f_r}{\partial x_i} \frac{\partial f_r}{\partial x_j}\con_r \right)x_j,
		\end{align*}
		so $\grad{\mf}{x}=Lx$ by \eqref{eq:lapdiff}. By a similar argument we also have
		\begin{equation*}
		\frac{\partial^2}{\partial x_i\partial x_j} \mf(x) = -L_{ij},
		\end{equation*}
		as desired.
	\end{proof}
	\label{prop:hesslap}
\end{prop}

\begin{cor}[Principal Minors Matrix-Tree Theorem]
	Let $\overline{\g}$ be the multigraph obtained by identifying all boundary nodes of $\g$ as a single vertex. If $\kg$ is the principle submatrix of $\lg$ with rows and columns indexed by $\inte$, then $\det(\kg)\in \R[\con_e : e\in E]$ is the generating polynomial of the set of spanning trees of $\overline{\g}$.
	
	\begin{proof}
		This follows from Proposition \ref{prop:hesslap} and \cite[Theorem 3.1]{varchenko2006}.
	\end{proof}
\end{cor}
	
\subsection{Discrete harmonic functions}
\label{sec:ddp}

Let $\con,\eng\in \C^k$ be indexed by $E$. We adopt the language of electrical networks, calling $\con$ the \emph{conductances} and $\eng$ the \emph{energies}. We also refer to the entries $\con_e$ (resp., $\eng_e$) collectively as the \emph{conductances} (resp., \emph{energies}). For more on electrical networks, see \cite[Chapter 3]{curtis2000}.

Recall the Laplacian matrix $\lg$ introduced in Example \ref{eg:gralap}. Let $\kg=\kg(\con)$ denote the submatrix of $\lg$ obtained by deleting all rows and columns indexed by $\B$. If $\con\in (0,\infty)^k$, then there is a unique minimizer of $x^T\lg x$ in $\{x\in \R^d : x_j= u(j)\mbox{ for all }j\in\B\}$. The minimizer $x$ is characterized by the equations
\begin{equation}
	\sum_{j\sim i} \con_{ij}(x_i-x_j)=0
	\label{eq:harm}
\end{equation}
for all $i\in \inte$. For $\con\in \C^k$, a point $x\in\C^d$ satisfying \eqref{eq:harm} for all $i\in\inte$ is called a \emph{harmonic function} on $(\g, u,\con)$. When there is no ambiguity, we say simply that $x$ is \emph{harmonic}.

If a harmonic function on $(\g, u,\con)$ exists, then it is unique; we denote it by $h(\con)$. A harmonic function exists unless $\kg$ is singular, which occurs only for $\con$ in a proper algebraic subset of $\C^k$. We say that $\con$ is \emph{generic} if $\kg$ is nonsingular. In particular, every $\con\in (0,\infty)^k$ is generic.

This discussion explains the name \emph{Dirichlet arrangement} for $\Au$; finding a harmonic function on $(\g,u,\con)$ is a discrete analog of the classical Dirichlet problem on a continuous domain, and in this analogy $u$ is the \emph{Dirichlet boundary data} (see \cite[Section 1.2]{curtis2000}).

\begin{eg}
	Let $\g$ be a path graph. Let $\B$ consist of both ends of the path, and write $\con=(\con_1,\ldots,\con_k)$ with the edges ordered from one end to the other. This graph is illustrated in Figure \ref{fig:path} with edges labeled by their conductances and boundary nodes marked by white circles.
	
	\begin{figure}[ht]
		\centering
		\input{tikz/path}
		\caption{A path graph with edge weights labeled and boundary nodes marked in white.}
		\label{fig:path}
	\end{figure}
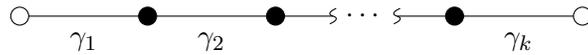
	
	Here the matrices $\lg$ and $\kg$ are symmetric and tridiagonal. For instance, when $k=5$ we have
	\begin{equation*}
		\kg =
		\begin{pmatrix}
			\con_1+\con_2&-\con_2&&\\
			-\con_2&\con_2+\con_3&-\con_3&\\
			&-\con_3&\con_3+\con_4&-\con_4\\
			&&-\con_4&\con_4+\con_5
		\end{pmatrix}.
	\end{equation*}
	One can use the recurrences in \cite{usmani1994} to compute $\kg^{-1}$ in terms of elementary symmetric polynomials: for $1\leq r\leq s\leq n$,
	\begin{equation}
		[\kg^{-1}]_{rs}=
		\frac{e_{r-1}(\con_1,\ldots,\con_i)e_{s-r}(\con_{i+1},\ldots,\con_j)e_{n-s}(\con_{j+1},\ldots,\con_k)}{e_n(\con_1,\ldots,\con_k)},
		\label{eq:grepath}
	\end{equation}
	where $e_0=1$. In particular, $\det \kg = e_n(\con_1,\ldots,\con_k)$, so $\con$ is generic in this example if and only if $e_n(\con_1,\ldots,\con_k)\neq 0$.
	
	When $\con=(1,\ldots,1)$ we have
	\begin{equation*}
		\kg =
		\begin{pmatrix}
			2 & -1 & & &\\
			-1&2&-1&&\\
			&-1&\ddots&\ddots&\\
			& &\ddots & \ddots &-1\\
			&&& -1 & 2
		\end{pmatrix}.
	\end{equation*}
	This matrix arises as the Cartan matrix of the root system $\mathsf{A}_n$, and as the matrix of coupling coefficients of $d$ harmonic oscillators in a linear chain (see, e.g., \cite[\S11.4]{humphreys1972} and \cite[Exercise 4.2]{mello2004}). It also plays a role in other boundary value problems on path graphs \cite{bendito2009,carmona2013}. Chung and Yau computed \eqref{eq:grepath} in this case:
	\begin{equation*}
		[\kg^{-1}]_{rs} = \frac{r(k-s)}{k}
	\end{equation*}
	for all $1\leq r,s\leq n$ \cite[Theorem 3]{chung2000}. 
	
	\label{eg:path}
\end{eg}

\subsection{Fixed-energy harmonic functions}
\label{sec:mfu}

We now connect master functions of $\Au$ to harmonic functions on $(\g, u)$. Given $\eng\in \C^k$, we write
\begin{align*}\efu&=\Psi_{\Au}\\
\mfu&=\Phi_{\Au}^{\eng}\\
\VVu &= \mathcal{V}(\Au,\eng).
\end{align*}
We continue to assume that $\g$ is connected and $\B$ is nonempty. If $\con$ is generic in the sense of Section \ref{sec:ddp}, then we write 
\[\efu(\con)=\efu(\con,h(\con)).\]

\begin{eg}
	Let $(\g,\B)=\comp{m}{n}$ as in Example \ref{eg:compjoin}. Fix positive integers $\ell_j$ for all $j\in \B$. Let $\eng\in \C^k$ be given for all $ij\in E$ by
	\[\eng_{ij}=\begin{cases}2&\mbox{if }i,j\in \inte\\-\ell_j &\mbox{if }j\in\B.\end{cases}\]
	Here we have
	\begin{equation*}
		\mfu(x) = \sum_{\{i,j\}\subset \inte} 2\log(x_i-x_j)-\sum_{\substack{i\in\inte \\ j\in\B}}  \ell_j\log (x_i-u(j)) .
	\end{equation*}
	This master function plays a crucial role in the construction of hypergeometric solutions of the $\mathfrak{sl}_2$ Knizhnik--Zamolodchikov equations \cite{mukhin2004,schechtman1991,scherbak2003}. Since the components of $\eng$ are not all positive, the structure of $\VVu$ is not settled by Proposition \ref{prop:varchenko}. In fact, the qualitative behavior of the critical points of $\mfu$ changes as $n$, $m$ and $\eng$ are allowed to vary. This is shown in \cite{scherbak2003} by characterizing the critical points of $\mfu$ in terms of polynomial solutions of Fuchsian differential equations.  
\end{eg}
	
	\begin{proof}[Proof of Theorem \ref{thm:cpchar}]
		Let $z\in \C^n$ extend $ u$, and fix $\eng\in (0,\infty)^k$. We must show that $z$ is $\eng$-harmonic on $(\g, u)$ if and only if $z\in \VVu$. Suppose first that $z$ is $\eng$-harmonic on $(\g, u)$. Then there is $\con\in \C^k$ such that $h(\con)=z$ and $\efu(\con)=\eng$. Thus for all $i\in \inte$ we have
		\begin{equation}
		0=\sum_{j\sim i} \con_{ij}(z_i-z_j) = \sum_{j\sim i} \frac{\eng_{ij}}{z_i-z_j} = \frac{\partial}{\partial x_i} \mfu(z),
		\label{eq:mainproof}
		\end{equation}
		where we have used the definition of $\efu$. Hence $z\in \VVu$.
		
		Conversely, suppose that $z\in \VVu$, and let $\gamma\in \C^k$ be given by $\gamma_{ij}=\eng_{ij}/(z_i-z_j)^2$ for all $ij\in E$. It is not hard to see that \eqref{eq:mainproof} holds again for all $i\in \inte$, so $h(\gamma)=z$ and moreover $\efu(\con)=\eng$. Hence $z$ is $\eng$-harmonic on $(\g, u)$.
	\end{proof}

	It seems likely that the following corollary is known in some form, given the extensive literature on electrical networks. However, we have not seen it stated as such.
	
	\begin{cor}
		Every point in every bounded chamber of $\Au$ is a harmonic function on $(\g, u,\con)$ for some choice of conductances $\con\in (0,\infty)^E$.
		\label{cor:everything}
		
		\begin{proof}
			This is an application of \cite[Theorem 3.3]{pressman2001} to Theorem \ref{thm:cpchar}.
		\end{proof}
	\end{cor}

	For a fixed $\eng\in \C^E$, as $\con$ ranges over the generic conductances with $\efu(\con)=\eng$, we call the functions $h(\con)$ the \emph{$\eng$-harmonic functions} on $(\g, u)$. The results of Abrams and Kenyon \cite{abrams2016} now follow:
	
	\begin{cor}[{\cite[Theorems 1--3]{abrams2016}}]
		Fix energies $\eng\in (0,\infty)^k$, and let $\acon$ be the set of all generic conductances $\con\in \C^k$ for which $\efu(\con)=\eng$. The following hold:
		\begin{enumerate}[(i)]
			\item $\acon\subset (0,\infty)^k$
			\item The $\eng$-harmonic functions on $(\g, u)$ form a system of distinct representatives for the bounded chambers of $\Au$.
		\end{enumerate}
	\begin{proof}
		Item (ii) is an immediate consequence of Proposition \ref{prop:centess}, Proposition \ref{prop:varchenko} and Theorem \ref{thm:cpchar}. Suppose that $\con\in \acon$. For all $ij\in E$ we have $\eng_{ij}=\con_{ij}(h_i(\con)-h_j(\con))^2 >0$, where we write $h_i(\con)$ for the $i$th component of $h(\con)$. Since $h(\con)$ is a real point, it follows that $\con_{ij}>0$, proving (i).
	\end{proof}
		\label{cor:abrams}
	\end{cor}

	In fact, there is a bijection from $\acon$ to the set of bounded chambers of $\Au$ for all $\eng$ outside a proper algebraic subset of $\C^k$. This follows, for instance, from a generalization of Proposition \ref{prop:varchenko} due to Orlik and Terao \cite{orlik1995}. Corollary \ref{cor:abrams} has applications in rectangular tilings \cite{abrams2016}. 
	
	\begin{proof}[Proof of Corollary \ref{cor:eharmcount}]
		This follows from Theorem \ref{thm:chambercount} and Corollary \ref{cor:abrams}(ii).
	\end{proof}

%% file: tikz/path.tex
\begin{tikzpicture}[xscale=1.7,yscale=1.7]
\coordinate (a) at (0,0);
\coordinate (b) at (1,0);
\coordinate (c) at (2,0);
\coordinate (c1) at (2.45,0);
\coordinate (c2) at (2.95,0);
\coordinate (d) at (3.4,0);
\coordinate (e) at (4.4,0);
\draw (a) -- (c1);
\draw (c2) -- (e);
\node () at (2.7,0) {$\cdots$};
\node[below] () at (0.5,-0.05) {$\con_1$};
\node[below] () at (1.5,-0.05) {$\con_2$};
\node[below] () at (3.9,-0.05) {$\con_k$};
\draw plot [smooth, tension = 1] coordinates { (2.475,0.06) (2.425,0.02) (2.475,-0.02) (2.425,-0.06)};
\draw plot [smooth, tension = 1] coordinates { ({2.475+0.5},0.06) ({2.425+0.5},0.02) ({2.475+0.5},-0.02) ({2.425+0.5},-0.06)};
\draw[fill=white] (a) circle (2pt);
\draw[fill=black] (b) circle (2pt);
\draw[fill=black] (c) circle (2pt);
\draw[fill=black] (d) circle (2pt);
\draw[fill=white] (e) circle (2pt);
\end{tikzpicture}

%% file: sections/galois.tex
\section{Totally real Galois action}
\label{sec:galois}

This appendix is dedicated to Theorem \ref{thm:gal} below, which generalizes an observation of Abrams and Kenyon \cite[Corollary 5]{abrams2016}. An algebraic number in $\R$ is called \emph{totally real} if all of its Galois conjugates over $\Q$ are real. The set $\Qtr$ of all totally real numbers is a subfield of $\R$, and the (infinite) extension $\Qtr/\Q$ is Galois.

\begin{thm}
	If $\AA$ is an essential real arrangement defined over $\Q$ and $\eng\in (0,\infty)^{\AA}$ is a rational point, then $\Gal(\Qtr/\Q)$ acts on the set of critical points of the master function $\mf$, and hence on the set of bounded chambers of $\AA$.
	
	\begin{proof}
		Let $k=|\AA|$. We have $x\in \VV$ if and only if $x$ satisfies \eqref{eq:cpeq} for all $i=1,\ldots,d$. Clearing denominators in \eqref{eq:cpeq} gives a system of polynomial equations over $\Q$:
		\begin{equation}
			\sum_{r=1}^k \frac{\partial f_r}{\partial x_i} \eng_r \prod_{s\neq r} f_s(x)=0.
			\label{eq:polysys}
		\end{equation}
		By Proposition \ref{prop:varchenko}, the system has only finitely many solutions $x\in M(\AA_\C)$, so each solution is an algebraic point.
		
		Let $\KK$ be the field generated over $\Q$ by $x_i$, as $x$ ranges over $\VV$ and $i$ ranges over $1,\ldots,d$. Replace $\KK$ by a Galois closure if necessary, and let $\sigma\in \Gal(\KK/\Q)$. Clearly if $x$ is a solution of the system \eqref{eq:polysys}, then $\sigma(x)$ is also a solution. Hence $\Gal(\KK/\Q)$ acts on $\VV$. Moreover, Proposition \ref{prop:varchenko} says that all solutions of \eqref{eq:polysys} are real, so $\KK\subset \Qtr$. The result follows.
	\end{proof}
	\label{thm:gal}
\end{thm}

	When $\AA=\Au$, Theorem \ref{thm:gal} gives an action of $\Gal(\Qtr/\Q)$ on the set $\com$ of compatible orientations of $(\g, u)$ for each rational point $\eng\in(0,\infty)^k$. Abrams and Kenyon conjectured in this case that if $\g$ is 3-connected, then the action is transitive given sufficiently general choices of $u$ and $\eng$ \cite[Conjecture 1]{abrams2016}. Theorem \ref{thm:gal} suggests that a similar statement might hold for any sufficiently ``robust'' arrangement $\AA$. Proposition \ref{prop:wheel} below describes an example in which $\g$ is 3-connected but the corresponding action is \emph{not} transitive.
	
	\begin{prop}
		Let $\g$ be a wheel graph on $d\equiv 3\pmod{4}$ vertices, and let $\B$ consist of 2 opposite vertices on the outer cycle of the wheel. Fix rational boundary data $u$, and let $\eng\in \C^k$ be identically 1. If $d>3$, then the action of $\Gal(\Qtr/\Q)$ on the set of $\eng$-harmonic functions on $(\g, u)$ is not transitive.
		\label{prop:wheel}
		
		\begin{proof}
			Label the outer vertices of $\g$ in a cycle by $i_0,\ldots,i_{d-2}$, and write $d=4\ell-1$. Without loss of generality, suppose that $\B=\{i_0,i_{2\ell-1}\}$ with boundary values $ u(i_0)=1$ and $ u(i_{2\ell-1})=-1$. We exhibit an $\eng$-harmonic function $f\in\Q^{n-2}$ on $(\g, u)$. Such a function is necessarily fixed by the action of $\Gal(\Qtr/\Q)$. Theorem \ref{thm:chambercount} gives $|\com|=(2\ell-1)^2$, so the result will follow.
			
			For $r=1,\ldots,\ell-1$ let
			\begin{equation}
				f(i_r)=\prod_{s=0}^{r-1} \frac{2(\ell-2s)-1}{2(\ell-2s)+1}.
				\label{eq:oddharm}
			\end{equation}
			For $r=\ell,\ldots,2\ell-2$ let $f(i_r)=-f(i_{2\ell-r-1})$, and for $r=2\ell,\ldots,4\ell-3$ let $f(i_r)=f(i_{4\ell-r-2})$. Finally, let $f$ be 0 at the center of the wheel. This defines a function $f\in \Q^n$. It is routine to verify that $f$ is $\eng$-harmonic on $(\g, u)$. The case $d=15$ is illustrated in Figure \ref{fig:wheel}.
		\end{proof}
	\end{prop}
	
		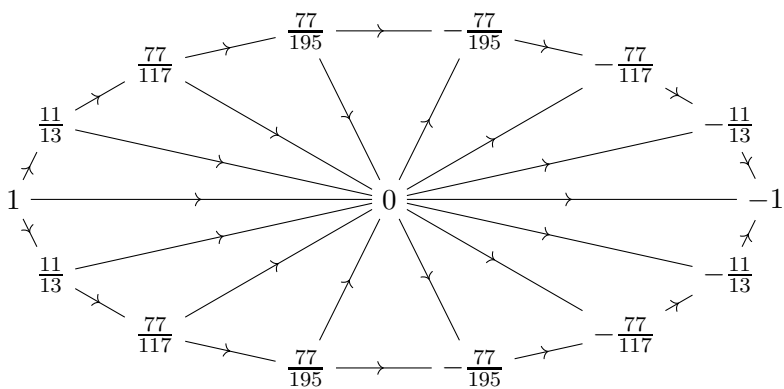
\begin{figure}[ht]
				\centering
				\input{tikz/wheel}
			\caption{A network with vertices labeled by the values of the function $f$ defined in the proof of Proposition \ref{prop:wheel} and edges oriented according to the associated compatible orientation.}
			\label{fig:wheel}
		\end{figure}

%% file: tikz/wheel.tex
\begin{tikzpicture}[xscale=5,yscale=2.3]
\tikzset{->-/.style={decoration={
			markings,
			mark=at position .5 with {\arrow{>}}},postaction={decorate}}}
\node (a) at (0:1) {$-1$};
\node (b) at (360/14:1) {$-\frac{11}{13}$};
\node (c) at (2*360/14:1) {$-\frac{77}{117}$};
\node (d) at (3*360/14:1) {$-\frac{77}{195}$};
\node (e) at (4*360/14:1) {$\frac{77}{195}$};
\node (f) at (5*360/14:1) {$\frac{77}{117}$};
\node (g) at (6*360/14:1) {$\frac{11}{13}$};
\node (h) at (7*360/14:1) {$1$};
\node (i) at (8*360/14:1) {$\frac{11}{13}$};
\node (j) at (9*360/14:1) {$\frac{77}{117}$};
\node (k) at (10*360/14:1) {$\frac{77}{195}$};
\node (l) at (11*360/14:1) {$-\frac{77}{195}$};
\node (m) at (12*360/14:1) {$-\frac{77}{117}$};
\node (n) at (13*360/14:1) {$-\frac{11}{13}$};
\node (o) at (0:0) {$0$};
\draw[->-] (b) -- (a);
\draw[->-] (c) -- (b);
\draw[->-] (d) -- (c);
\draw[->-] (e) -- (d);
\draw[->-] (f) -- (e);
\draw[->-] (g) -- (f);
\draw[->-] (h) -- (g);
\draw[->-] (h) -- (i);
\draw[->-] (i) -- (j);
\draw[->-] (j) -- (k);
\draw[->-] (k) -- (l);
\draw[->-] (l) -- (m);
\draw[->-] (m) -- (n);
\draw[->-] (n) -- (a);
\draw[->-] (o) -- (a);
\draw[->-] (o) -- (b);
\draw[->-] (o) -- (c);
\draw[->-] (o) -- (d);
\draw[->-] (o) -- (l);
\draw[->-] (o) -- (m);
\draw[->-] (o) -- (n);
\draw[->-] (e) -- (o);
\draw[->-] (f) -- (o);
\draw[->-] (g) -- (o);
\draw[->-] (h) -- (o);
\draw[->-] (i) -- (o);
\draw[->-] (j) -- (o);
\draw[->-] (k) -- (o);
\end{tikzpicture}

%% file: sections/thanks.tex
\section*{Acknowledgments}

\noindent The author thanks Trevor Hyde and Jeffrey Lagarias for helpful comments on drafts of the paper.